\documentclass[11pt]{amsart}
\usepackage{mathtools,amsfonts,amsmath,amssymb,enumerate,amsthm,amscd,cite,tikz}
\usetikzlibrary{cd}

    \newcommand{\BA}{{\mathbb {A}}} 
    \newcommand{\BC}{{\mathbb {C}}}

    \newcommand{\BQ}{{\mathbb {Q}}} \newcommand{\BR}{{\mathbb {R}}}

     \newcommand{\BZ}{{\mathbb {Z}}}

     \newcommand{\CF}{{\mathcal {F}}}

     \newcommand{\CN}{{\mathcal {N}}}
    \newcommand{\CO}{{\mathcal {O}}} 
     
    \newcommand{\CS}{{\mathcal {S}}}

  \newcommand{\cal}{\mathcal}

    \newcommand{\diag}{{\mathrm{diag}}}

     \newcommand{\GL}{{\mathrm{GL}}}
    
    \newcommand{\Hom}{{\mathrm{Hom}}}
    
    \renewcommand{\Im}{{\mathrm{Im}}}

    \newcommand{\Res}{{\mathrm{Res}}}

    \newcommand{\SL}{{\mathrm{SL}}}
     
    \newcommand{\Sp}{{\mathrm{Sp}}}

    \newcommand{\tr}{{\mathrm{tr}}}

    \newcommand{\vol}{{\mathrm{vol}}}

    \newcommand{\matrixx}[4]{\begin{pmatrix}
#1 & #2 \\ #3 & #4
\end{pmatrix} }        

\newtheorem{thm}{Theorem}[subsection]
\newtheorem{defn}[thm]{Definition}
\newtheorem{prop}[thm]{Proposition}
\newtheorem{rmk} [thm]{Remark}
\newtheorem{lem}[thm]{Lemma}

    \newcommand{\pair}[1]{\langle {#1} \rangle}

    \newcommand{\Gr}{\mathrm{Gr}}

    \numberwithin{equation}{section}

\newcommand{\la}{\langle}
\newcommand{\ra}{\rangle}

\begin{document}

\title{Theta Liftings for $(\GL_n,\GL_n)$ Type Dual Pairs of Loop Groups}
\author[Yanze Chen, Yongchang Zhu]{Yanze Chen$^*$ and Yongchang Zhu}
\address{Department of Mathematics, The Hong Kong University of Science and Technology, Clear Water Bay, Kowloon, Hong Kong}

\email{ychenen@connect.ust.hk}
\email{mazhu@ust.hk}

\thanks{$*$This research is supported by Hong Kong RGC grant 16301718.}
\maketitle
\begin{abstract}
	In this article we prove the theta liftings of a cusp form on the loop group $\GL_n$ induced from a classical cusp form for the loop group ``dual pair'' $(\GL_n,\GL_n)$ is an Eisenstein series.
\end{abstract}

\section{Introduction   }\label{S1}

The method of theta lifting between reductive dual pairs via Weil representations provides a very useful tool and has wide applications in representation theory and automorphic forms.
In the global theory, one considers the lifting of an irreducible automorphic cuspidal representation $\pi$ of a classical group $G'$ to its dual group $G$.
The theta lifting for automorphic forms on loop groups was first considered in \cite{GZ1} \cite{GZ2}, where a version of Siegel-Weil formula is proved, it can be considered
as a lifting of the trivial automorphic forms to Eisenstein series. In \cite{LZ2}, Rallis constant term formula and the tower properties are
 proved for the lifting of loop orthogonal group to loop
 sympletic group. In this work, we consider $( \GL_n , \GL_n) $ dual pair. When $n\geq2$, we prove that a cuspidal automorphic form $\varphi$ of classical group $\GL_n$ lifts to the
 an Eisenstein of loop $\GL_n $ induced from $\varphi$.

 To state our main result, we briefly introduce some notations and background.
 For a finite dimensional symplectic vector space $W$ over the field of rational numbers ${\BQ}$ we denote
\[
W((t))={W}\otimes_{\BQ}{\BQ}((t)),\quad W [[t]]=W\otimes_{\BQ}{\BQ}[[t]],\quad {W}[t^{-1}]t^{-1}= W\otimes_{\BQ} {\BQ}[t^{-1}]t^{-1}.
\]
Let  ${\BA}$ be the adele ring of ${\BQ}$. We fix a non-trivial additive character $\psi:\BQ\backslash\BA\to U(1)$.  Similar notations apply for certain adelic spaces with ${\BQ}((t))$ replaced by
 ${\BA}\langle t\rangle$, which is a certain restricted product of
local Laurent series ${\BQ}_v((t))$ (see Section 2.2). We also need ${\BQ} \langle t\rangle:=
{\BQ}((t))\cap {\BA} \langle t\rangle$.
 Let $(G,G')$ be a reductive dual pair over ${\BQ}$ sitting inside $\Sp({W})$,
 Following \cite{LZ1}, we have the adelic loop symplectic group $\widetilde{\mathrm{Sp}}( {W}_{\BA}\langle t\rangle)$, which is a central extension of $\mathrm{Sp}( {W}_{\BA}\langle t\rangle)$ and can be written as a certain restricted product of local loop groups. Its Weil representation, denoted by $\omega$, is realized on a certain space $\mathcal{S}'({ V}_{-,\BA})$ of Bruhat-Schwartz functions,
  where ${ V_-}:={ W}[t^{-1}]t^{-1}$ is a Larangian space of ${W}\langle t\rangle$. An important new feature in the loop setting is that in the definition of $\cal{S}'({V}_{-,\BA})$ one has to invoke a parameter of loop rotations $q: t\mapsto qt$, where $q\in{\BA}^\times$ has idele norm $|q|>1$.  Then the theta functional for $f\in \mathcal{S}'({ V}_{-,\BA})$,
\[
\theta(f):=\sum_{r\in { V}_{-,F}}f(r)
\]
converges absolutely \cite{Z, LZ1}, and $g\mapsto \theta(\omega(g)f)$ defines an automorphic function on the arithmetic quotient $\rm{Sp}({W}\langle t\rangle)\backslash \widetilde{\rm{Sp}}({W}_{\BA}\langle t\rangle)$.
\par In this work, we consider the type-two dual pair $(\GL_n , \GL_n )$ in $\rm{Sp}({W})$, so ${\dim} \, W = 2 n ^2 $. We assume $n\geq2$ in this introduction, although we also considered the case $n=1$ in Section 5.  Let
 $ \GL_n ({\BA})^1$ be the subgroup of $\GL_n ({\BA})$ consisting of elements $g \in \GL_n ({\BA})$ with condition
 $ | {\rm det} ( g ) |=1 $. It contains $\GL_n ( {\BQ} )$ because of the product formula. Let
\[ \GL_n ( {\BA} \la t \ra_+  )^1 = \{ g \in \GL_n ( {\BA} \la t \ra_+  ) \, | \,   g ( 0) \in \GL_n ({\BA})^1\} \]
 Let $\varphi$ be an irreducible automorphic cusp form on  $\GL_n ( {\BQ} ) \backslash \GL_n ({\BA})^1$,
 We pull back $\varphi$ to $\GL_n ({\BA}\langle t\rangle_+)^1 $ via the evaluation map
 $\GL_n ({\BA}\langle t\rangle_+)^1 \to \GL_n ({\BA})^1 $ at $t=0$,
  where ${\BA}\langle t\rangle_+={\BA}\langle t\rangle\cap {\BA}[[t]]$, and define the theta lifting to
  $\widetilde\GL_n ( {\BA} \la t \ra )$ by
\begin{equation}\label{1.1}
\Theta_f(\varphi,g)=\int_{\GL_n ( {\BQ} \langle t\rangle_+)\backslash \GL_n ({\BA}\langle t\rangle_+)^1 }
   \theta(\omega(h)\omega(g)f)\varphi(h)dh,\quad g\in \GL_n ({\BA} \langle t\rangle),
\end{equation}
noting that the quotient space $\GL_n ( {\BQ} \langle t\rangle_+)\backslash \GL_n ({\BA}\langle t\rangle_+)^1$ is locally compact. We can show that  the theta function is of polynomial growth hence the theta integral converges thanks to the rapid decay of cusp forms. We also remark that in defining the theta lifting (\ref{1.1}) we take the integral only over $\GL_n ( {\BQ} \langle t\rangle_+)\backslash \GL_n ({\BA}\langle t\rangle_+)^1$ instead of $\GL_n ( {\BQ} \langle t\rangle_+)\backslash \GL_n ({\BA}\langle t\rangle_+)$,
the main reason
 is that the subgroups $\widetilde \GL_n  ({\BA} \langle t\rangle)  $ and $ \GL_n ({\BA}\langle t\rangle_+)$  in $ \widetilde{\rm Sp} ( W_{\BA} \la t \ra )$ are
  not commutative. But ${\GL_n } ({\BQ} \langle t\rangle)  $ commutes with $\widetilde \GL_n  ({\BA} \langle t\rangle_+ )^1  $.
\par On the other hand, consider the function $I_f(\varphi,g)$ on $\widetilde \GL_n(\BA\pair{t})$ given by $$I_f(\varphi,g)=\int_{\GL_n(\BA)^1}\left(\int_{M_n(\BA)}(\pi(g)f)(0,yt^{-1})\psi(\tr(ay^t))dy\right)\varphi(a)da$$ for a suitable normalization of Haar measures, where $y^t$ is the transpose of $y$. For $g$ in the subgroup $\GL_n(\BA)$ of $\widetilde\GL_n(\BA\pair{t})$ the function $I_f(\varphi,g)$ is the classical theta lifting of the cusp form $\varphi$ with respect to $f$ in the dual pair $(\GL_n,\GL_n)$. The function $I_f$ is left $\GL_n(\BQ\pair{t}_+)$-invariant, so we define the Eisenstein series induced from $I_f$ by $$E_f(\varphi,g)=\sum_{r\in\GL_n(\BQ\pair{t}_+)\backslash\GL_n(\BQ\pair{t})}I_f(\varphi,g)$$
The main result of this paper is 
\begin{thm}
	With the above notations, we have $$\Theta_f(\varphi,g)=E_f(\varphi,g)$$ for $g\in \widetilde\GL_n(\BA\pair{t})$. 
\end{thm}
\

This theorem is proved using another model of the representation of $\widetilde\GL_n ( F(( t)) )$ on ${\cal S} ( V_-)$ for local fields $F$  
which is essentially in \cite{Ka}. 

\

The organization  of this paper is as follows. In Section 2, we briefly recall the local and global theory of Weil representations and theta functionals for loop symplectic groups following \cite{Z, LZ1}.  In Section 3, we construct a different model of
 the representation of
 ${\widetilde \GL_n} ( {\BA} )$ and prove that is isomorphic to  ${\cal S}' ( V_{-,\BA} ) $.
 In Section 4, we define the theta liftings for the $(\GL_n,\GL_n)$ dual pair of loop groups. In Section 5 we prove that when $n=1$ the theta lifting of a constant function can be interpreted as an Eisenstein series. In Section 6 we treat the case $n>1$ and prove the main result. 

{\bf Acknowledgement.} The authors would like to thank Howard Garland for pointing out the isomorphism $\Phi $ of representations of $\widetilde\GL_n(F((t)))$.

\

\section{Weil Representation  }\label{S2}
In this section we briefly recall the Weil representations of symplectic loop groups over local and global fields, as well as the absolute convergence of theta functionals, following \cite{Z} and \cite{LZ1}.

\subsection{Local theory} \label{2.1} Let $F$ be a field of characteristic 0, and $W$ be a symplectic vector space of dimension $2n$ over $F$, with symplectic form denoted by $\langle,\rangle_F$. It gives rise to an $F((t))$-valued symplectic form $\langle,\rangle_{F((t))}$ on $W((t))$ by scalar extensions. It further gives
an $F$-valued symplectic form $\langle,\rangle$ on $W((t))$ by taking the residue
\begin{equation}\label{form}
\langle w, w'\rangle=\textrm{Res }\langle w, w'\rangle_{F((t))},
\end{equation}
where $w, w'\in W((t))$, and Res $a$ for $a\in F((t))$ is the coefficient of $t^{-1}$ in $a$.
 We set
\begin{equation}\label{lag}
  V_-=W[t^{-1}]t^{-1}, \quad V_+=W[[t]]
\end{equation}
are maximal isotropic subspaces of $W((t))$.  Denote by
\begin{equation}\label{pro}
p_-: W((t))\to V_-, \quad p_+: W((t)) \to V_+
\end{equation}
the natural projections. Let $\textrm{Sp}(W((t)), V_+)$ be the group of all $F$-linear symplectic isomorphisms $g$ of $W((t))$ such that $V_+g$ and $V_+$ are commensurable, which contains $\textrm{Sp}_{2n}(F((t)))$ as a subgroup. Here by convention $\textrm{Sp}(W((t)),V_+)$ acts on $W((t))$ from the right. Then with respect to the decomposition $W((t))=V_-\oplus V_+$, each $g\in\textrm{Sp}(W((t)), V_+)$ can be represented by a matrix
\begin{equation} \label{matrix}
g=\begin{pmatrix} a_g & b_g \\  c_g & d_g\end{pmatrix}.
\end{equation}
Note that the image of $c_g: V_+\to V_-$ is a finite dimensional space over $F$.
The group law of the Heisenberg group $H=W((t))\times F$ is defined by
\[
(w, z)(w',z')=(w+w', \frac{1}{2}\langle w, w'\rangle+z+z').
\]
Then $\textrm{Sp}(W((t)),V_+)$ acts on $H$ from the right by $(w,z)\cdot g=(w\cdot g, z)$.

Now assume further that $F$ is a local field. 
Let $\mathcal{S}(V_-)$ be the space of complex valued Schwartz functions on $V_-$, i.e. functions whose restriction to each finite dimensional subspace is a Schwartz function in the usual sense.
Typical examples include the characteristic function of $t^{-1}\mathcal{O}[t^{-1}]^{2n}$ if $F$ is a $p$-adic field with ring of integers $\mathcal{O}$, and the Gaussian
function $e^{iQ(x)}$ if $F$ is archimedean and $Q(x)$ is a complex valued quadratic form on  $V_-$ with positive definite imaginary part.
Fix a non-trivial additive character $\psi$ of $F$. Then the Heisenberg group $H$ acts on $\mathcal{S}(V_-)$ in the usual way such that the central element $(0,z)$ acts by the scalar
$\psi(z)$.
For $g\in \textrm{Sp}(W((t)), V_+)$ with decomposition (\ref{matrix}) and a choice of Haar measure $\mu$  on $\textrm{Im }c_g$, define an operator $T(g)$ on $\mathcal{S}(V_-)$ by
\begin{equation}\label{action}
(T(g)\phi)(x)=\int_{\textrm{Im }c_g}S_g(x+y)\phi(xa_g+ yc_g) \mu (  yc_g),
\end{equation}
where
\[
S_g(x+y)=\psi\left(\frac{1}{2}\langle xa_g, xb_g\rangle+\frac{1}{2}\langle yc_g, yd_g\rangle+\langle yc_g, xb_g\rangle\right).
\]
In particular, if $c_g=0$ then
\begin{equation}\label{gamma0}
(T(g) \phi)(x)=\psi\left(\frac{1}{2}\langle xa_g, xb_g\rangle\right)\phi(xa_g).
\end{equation}
The operator $T(g)$ is compatible with the Heisenberg group action, i.e. for $h\in H$ one has
\[
T (g)^{-1}h T(g) =h\cdot g.
\]
It was proved in \cite{Z} that $T (g_1)T (g_2)$ coincides with $T(g_1g_2)$ up to a scalar, and $g\mapsto T(g)$ gives a projective representation of $\textrm{Sp}(W((t)),V_+)$ on $\mathcal{S}(V_-)$.  By restriction, we obtain a projective representation of $\textrm{Sp}_{2n}(F((t)))$ on $\mathcal{S}(V_-)$.
The associated central extension is denoted by $\widetilde{\textrm{Sp}}_{2n}(F((t)))$, the symbol of the central extension is computed
 in \cite{Z}, its formula will not be used in this work.
We call the representation the Weil representation, denoted by $\left(\omega, \mathcal{S}(V_-)\right)$.

To specify the dependance of $T(g) $ on the the Haar measure $\mu $, we denote it by $ T( g , \mu ) $.
 We will denote the group generated by $T(g , \mu) $ for $ g \in \GL_n ( F (( t)) ) $, any $\mu $ on $\textrm{Im }c_g$
 by $ \widetilde\GL_n ( F (( t)) ) $.  The map $ T( g , \mu ) \mapsto g $ is a homomorphism from  $ \widetilde\GL_n ( F (( t)) ) $
  to ${\GL_n} ( F (( t)) )$ with kernel ${\BR}^\times $. Since we consider only $(\GL , \GL )$-dual pair, only
    $ \widetilde\GL_n( F (( t)) ) $ and its adelic version appear in this work.

We also need to introduce the group of ``rotations of loops" which is isomorphic to $F^\times$.
 It acts on $F((t))$ and the formal 1-forms $F((t))dt$  by
\[
 a (t) \cdot\lambda =a(\lambda^{-1} t) \textrm{\quad and \quad} a(t)dt\cdot \sigma(t)= \lambda^{-1} a(\lambda^{-1} t )dt
\]
 Assume that the symplectic form $\langle,\rangle$ on $F^{2n}$ is represented by the matrix $\begin{pmatrix} 0 & I \\ -I & 0\end{pmatrix}$. View the first $n$ components of $F((t))^{2n}$ as elements in $F((t))$, and the last $n$ components as 1-forms in $F((t))dt$ (without writing $dt$).  Then the action of the rotation group $F^\times $ on $W((t))$ preserves the symplectic form, it can be viewed as a subgroup of $\textrm{Sp}(W((t)),V_+)$. Thus it is clear that elements of $F^\times$ act on $\mathcal{S}(V_-)$ by the formula (\ref{gamma0}). The group $F^\times $ also acts on $\textrm{Sp}_{2n}(F((t)))$ as automorphisms.

We can define a ``maximal compact subgroup" $K$ of $\widetilde{\textrm{Sp}}_{2n}(F((t)))$ so that its intersection with the center $\BC^\times$ is $S^1$, and its image $K'$ in $\textrm{Sp}_{2n}(F((t)))$ is as follows. If $F$ is a $p$-adic field with ring of integers $\mathcal{O}$, then $K'=\textrm{Sp}_{2n}(\mathcal{O}((t)))$; if $F$ is archimedean, then
\[
K'=\{g\in \textrm{Sp}_{2n}(F[t,t^{-1}]): \overline{g(t)}g(t^{-1})^T=I_{2n}\},
\]
where $\overline{(\cdot)}$ stands for the complex conjugate, and $(\cdot)^T$ is the matrix transpose.
We recall the following result from \cite{GZ2}.

\begin{lem}\label{fun}
\begin{enumerate}[(i)]
	\item If $F$ is non-archimedean with odd residue characteristic and ring of integers $\mathcal{O}$, and if the conductor of $\psi$ is $\mathcal{O}$, then the characteristic function
$\phi_0$ of $t^{-1}\mathcal{O}[t^{-1}]^{2n}$ is fixed by $K$.	
	\item If $F$ is archimedean, then there is a nonzero element $\phi_0\in \mathcal{S}(V_-)$ fixed by $K$ up to a scalar.

\end{enumerate}
\end{lem}

\subsection{Global theory} From now on we assume that $F$ is ${\BQ}$. Thus for each place $v$ of ${\BQ}$ one has the local Weil representation $\omega_v$ of $\widetilde{\textrm{Sp}}_{2n}({\BQ}_v((t)))\rtimes
 {\BQ}_v^\times $ on $\mathcal{S}(V_{-, v})$, where
\[ V_{-, v}=W_v[t^{-1}]t^{-1},\quad W_v=W\otimes_{\BQ} {\BQ}_v.\]
Let ${\BA}$ be the adele ring of ${\BQ}$, and $\psi=\prod_v\psi_v$ be a non-trivial character of ${\BA}/{\BQ}$. Define
\[
{\BA}\langle t\rangle =\left\{(a_v)\in\prod_v {\BQ}_v((t)): a_v\in {\BZ}_v((t))\textrm{ for almost all finite places }v\right\},
\]
and ${\BQ}\langle t\rangle= {\BQ}((t))\cap {\BA}\langle t\rangle$, which is a subfield of ${\BQ}((t))$. The adelic metaplectic loop group
   $\widetilde{\textrm{Sp}}_{2n}({\BA}\langle t\rangle)$
   is then the restricted product $\prod'_v \widetilde{\textrm{Sp}}_{2n}({\BQ}_v((t)))$ with respect to the ``maximal compact" subgroups $K_v$, and there is a short exact sequence
\begin{equation}\label{ade}
1\longrightarrow \bigoplus_v {\BC}^\times\longrightarrow \widetilde{\textrm{Sp}}_{2n}({\BA}\langle t\rangle)\longrightarrow \textrm{Sp}_{2n}({\BA}\langle t\rangle)\longrightarrow 1.
\end{equation}
The idele group ${\BA}^\times $ acts on $\widetilde{\textrm{Sp}}_{2n}({\BA}\langle t\rangle)$ by rotation of $t$.

For almost all finite places $v$ there is $\phi_{0,v}\in \mathcal{S}(V_{-,v})$ fixed by $K_v$, hence the restricted tensor product $\mathcal{S}(V_{-,\BA}):=\bigotimes'_v \mathcal{S}(V_{-,v})$ with respect to $\{\phi_{0,v}\}$ is a representation of $\widetilde{\textrm{Sp}}_{2n}(\BA\langle t\rangle)$, called the global Weil representation. Let us denote this representation by $\omega=\bigotimes'_v\omega_v$.
We need to further introduce a suitable subrepresentation, on which the theta functional
\begin{equation}\label{theta}
\phi\mapsto \theta(\phi)=\sum_{r\in V_{-,F}}\phi(r)
\end{equation}
defined in the usual way, is absolutely convergent.

For a finite place $v$, a subgroup of the Heisenberg group $H_v=F_v((t))^{2n}\times F_v$ is called a congruence subgroup if it contains $\varpi_v^k\mathcal{O}_v((t))^{2n}$ for some integer $k$, where $\varpi_v\in\mathcal{O}_v$ is a local parameter. A function $\phi_v\in \mathcal{S}(V_{-,v})$ is called elementary if it is bounded and fixed by a congruence subgroup of $H_v$. For example the function $\phi_0$ in Lemma \ref{fun} (i) is elementary. Let $\mathcal{E}(V_{-,\BA})$ be the space of functions on $V_{-,\BA}$ which are finite linear combinations of $\omega(g)\prod'_v \phi_v$, where
$g\in\widetilde{\textrm{Sp}}_{2n}(\BA\langle t\rangle)$, and the following hold:

\begin{itemize}

\item $\phi_v=\phi_{0,v}$ for almost all finite places $v$;

\item $\phi_v$ is elementary for each remaining finite place $v$;

\item $\phi_v=P_v\cdot \phi_{0,v}$ for each infinite place $v$, where $P_v$ is a polynomial function on some finite dimensional subspace $W_v[t^{-1},\ldots, t^{-N}]$ of $X_v$.

\end{itemize}
It is clear that $\mathcal{E}(V_{-,\BA})$ is a subrepresentation of $\mathcal{S}(V_{-,\BA})$.
We also have to introduce a sub-semigroup of ${\BA}^\times_{>1} $ by
consisting of ideles $q \in {\BA}^\times$ with $ |q | > 1 $.
Then we let
\begin{equation}
\mathcal{S}'(V_{-,\BA})={\BA}^\times_{>1} \cdot \mathcal{E}(V_{-,\BA})
\end{equation}
By the  result of \cite{LZ1}, we have

\begin{thm}\label{conv}
If $\phi \in  \mathcal{S}'(V_{-,\BA})$, then the theta series $\theta(\phi)$ converges absolutely and is invariant under $\mathrm{Sp}_{2n}({\BQ} \langle t\rangle)$.
\end{thm}

\

\section{An Alternative Model.}\label{S3}

In this section we construct a different model of the representation ${\cal S} ( V_- )$ of $\widetilde{\GL_n} ( F(( t))) $
 for a local field $F$, which is due to \cite{Ka}. 
 The local models fit to an adelic model. It will be used to compute the theta lifting.

\subsection{Lattices}
Let $F$ be a local field, $L_0=F^n[[t]]$ in $X=F^n((t))$. As before vectors in $X$ are row vectors. A \textbf{lattice} $L$ in $F((t))^n$ is a free $F[[t]]$-submodule of rank $n$ in $F((t))^n$ that is comeasurable with $L_0$ i.e. \begin{equation}
	\dim_FL_0/(L_0\cap L)<\infty,\,\dim_F L/(L_0\cap L)<\infty
\end{equation} Let $\Gr_n(F)$ be the set of all lattices in $F((t))^n$. $\GL_n(F((t)))$ acts transitively on $\Gr_n(F)$ on the right, the stabilizer of $L_0$ is the subgroup $\GL_n(F[[t]])$. This gives a bijection of sets \begin{equation}
	\Gr_n(F)\cong\GL_n(F[[t]])\backslash \GL_n(F((t)))
\end{equation}
\begin{lem}
	For any $L\in\Gr_n(F)$, there exists $M,N>0$ such that \begin{equation}
		t^{M}F[[t]]^n\subseteq L\subseteq t^{-N}F[[t]]^n
	\end{equation}
\end{lem}
\par For each pair of lattices $L_1,L_2$ in $\Gr_n(F)$ we define an $\BR_{>0}$-torsor $\mu(L_1,L_2)$ as follows: First, for lattices $L_1\subseteq L_2$, we define $\mu(L_1,L_2)$ to be the $\BR_{>0}$-torsor of Haar measures on $L_2/L_1$. For general lattices $L_1,L_2$, there exists lattice $L$ s.t. $L\subseteq L_1$ and $L\subseteq L_2$, and we define $\mu(L_1,L_2)$ to be $$\mu(L,L_1)^*\times_{\BR_{>0}}\mu(L,L_2)$$ where $\mu(L,L_1)^*=\Hom_{\BR_{>0}}(\mu(L,L_1),\BR_{>0})$ is the dual $\BR_{>0}$-torsor, and for two $\BR_{>0}$-sets $M,N$, we define $M\times_{\BR_{>0}}N$ to be the set of equivalence classes of $M\times N$ with equivalence relation generated by $(r\cdot m,r^{-1}\cdot n)=(m,n)$ for $r\in\BR_{>0},\,m\in M,\,n\in N$. It has a natural $\BR_{>0}$-action by $r\cdot[m,n]:=[r\cdot m,n]$. If $M,N$ are $\BR_{>0}$-torsors, so is $M\times_{\BR_{>0}}N$.
\begin{lem}
	$\mu(L_1,L_2)$ is well-defined, namely for any $L,L'\subseteq L_1\cap L_2$, there exists a canonical isomorphism $$\mu(L,L_1)^*\times_{\BR_{>0}}\mu(L,L_2)\cong \mu(L',L_1)^*\times_{\BR_{>0}}\mu(L',L_2) $$
\end{lem}
\par In particular, for any lattice $L\in\Gr_n(F)$, $\mu(L,L)$ is canonically isomorphic to $\BR_{>0}$.
\begin{proof}
	We will use the following lemma, whose proof is straightforward: 
	\begin{lem}
	Let $0\to V_1\xrightarrow{i} V_2\xrightarrow{p} V_3\to 0$ be a short exact sequence of finite-dimensional $F$-vector spaces. Let $dv_1$ be a Haar measure on $V_1$, $dv_3$ be a Haar measure on $V_3$. For any Schwartz function $f\in\CS(V_2)$, let $p_!f$ be the function on $V_3$ defined by integration on the fibre $$(p_!f)(v_3)=\int_{V_1}f(\widetilde v_3+v_1)dv_1$$ where $\widetilde v_3$ is an arbitrary element in $p^{-1}(v_3)$. Then $p_!f\in\CS(V_3)$, and the functional $$f\mapsto \int_{V_3}(p_!f)(v_3)dv_3$$ defines a Haar measure on $V_2$. If we let $\mu(V)$ be the $\BR_{>0}$-torsor of Haar measures on $V$ for every finite-dimensional $F$-vector space $V$, then the above construction gives a canonical isomorphism $$\mu(V_2)\cong \mu(V _1)\times_{\BR_{>0}}\mu(V_3).$$
\end{lem} With this lemma, it is easy to see that for any lattice $L\subseteq L_1\cap L_2$, we have $$\mu(L,L_1)^*\times_{\BR_{>0}}\mu(L,L_2)\cong \mu(L_1\cap L_2,L_1)^*\times_{\BR_{>0}}\mu(L_1\cap L_2,L_2)$$
\end{proof}

\begin{prop}
	For lattices $L_1,L_2,L_3$ in $\Gr_n(F)$, we have a canonical isomorphism of $\BR_{>0}$-torsors \begin{equation}
		\mu(L_1,L_2)\times_{\BR_{>0}}\mu(L_2,L_3)\xrightarrow{\sim}\mu(L_1,L_3)
	\end{equation} In particular, $\mu(L_1,L_2)\times_{\BR_{>0}}\mu(L_2,L_1)$ is canonically isomorphic to $\BR_{>0}$.
\end{prop}
\begin{proof}
Take a lattice $L\subseteq L_1 \cap L_2 \cap L_3$, then we have $$\mu(L_1,L_2)\cong\mu(L,L_1)^*\times_{\BR_{>0}}\mu(L,L_2)$$ $$\mu(L_2,L_3)\cong\mu(L,L_2)^*\times_{\BR_{>0}}\mu(L,L_3)$$ $$\mu(L_1,L_3)\cong\mu(L,L_1)^*\times_{\BR_{>0}}\mu(L,L_3)$$ The isomorphism then follows from the duality between $\mu(L,L_2)$ and $\mu(L,L_2)^*$.
\end{proof}
\subsection{Schwartz Functions.}
For $L_1,L_2\in\Gr_n(F)$ with $L_1\subseteq L_2$, we define $\CS(L_1,L_2)$ to be the $\BC$-vector space \begin{equation}
	\CS(L_2/L_1)\times_{\BR_{>0}}\mu(L_1,L_0)
\end{equation} where $\CS(L_2/L_1)$ is the space of Schwartz functions on the finite-dimensional $F$-vector space $L_2/L_1$, on which $\BR_{>0}$ acts naturally by scalar multiplications. An element in $\CS(L_1,L_2)$ will be denoted $f\otimes\mu$ with $f\in\CS(L_2/L_1)$ and $\mu\in\mu(L_1,L_0)$.
\par For lattices $L_1\subseteq L_2\subseteq L_3$ in $\Gr_n(F)$ we have a short exact sequence of finite-dimensional $F$-vector spaces \begin{equation}
	0\to L_2/L_1 \xrightarrow{i}L_3/L_1\xrightarrow{p}L_3/L_2\to 0
\end{equation} and we define two $F$-linear maps $i^*:\CS(L_1,L_3)\to\CS(L_1,L_2)$ and $p_*:\CS(L_1,L_3)\to\CS(L_2,L_3)$ as follows: for $f(x)\otimes\mu\in\CS(L_1,L_3)$, define $$i^*(f\otimes\mu)=f(i(x))\otimes\mu$$ As for $p_*$, under the canonical isomorphism of Proposition 1.1.3 we can write $\mu=\mu_1\mu_2$ with $\mu_1\in\mu(L_1,L_2),\,\mu_2\in\mu(L_2,L_0)$, then we define $p_*(f\otimes\mu)=g\otimes\mu_2$ where $g\in\CS(L_3/L_2)$ is given by $$g(x)=\int_{L_2/L_1}f(x+y)d\mu_1(y)$$ This is well-defined because a different choice of the decomposition of $\mu$ only involves a scalar dilation, which results in a simultaneous dilation in the integration and in the measure.
\par A family $(\phi_{L,L'})_{L\subset L'}$ indexed by pairs of lattices $L,L'\in\Gr_n(F)$ with $L\subseteq L'$ is called a \textbf{compatible family} if for any lattices $L_1,L_2,L_3\in\Gr_n(F)$ with $L_1\subseteq L_2\subseteq L_3$, let $i,p$ be the natural maps in the short exact sequence (3.6), then we have \begin{equation}
	i^*(\phi_{L_1,L_3})=\phi_{L_1,L_2},\,p_*(\phi_{L_1,L_3})=\phi_{L_2,L_3}
\end{equation} Let $\CS(X,L_0)$ be the set of compatible families $(\phi_{L,L'})_{L\subseteq L'}$, called the space of \textbf{Schwartz functions} of $(X,L_0)$. It has an obvious $\BC$-vector space structure coming from the $\BC$-vector space structures on the various $\CS(L,L')$.
\par A subset $\CN=\{(L,L')\in\Gr_n\times\Gr_n:L\subseteq L'\}$ is called a \textbf{net} if for any pair $L_1,L_2\in\Gr_n$ with $L_1\subseteq L_2$, there exists $(L,L')\in N$ such that $L\subseteq L_1\subseteq L_2\subseteq L'$. A compatible family is determined by a compatible family parameterized by $\CN$. To see this, suppose we have a compatible family $(\phi_{L,L'}\otimes\mu_{L,L'})_{(L,L')\in\CN}$ parameterized by $\CN$. For any pair $L_1\subseteq L_2$ of lattices, there exists $(L,L')\in\CN$ such that $L\subseteq L_1\subseteq L_2\subseteq L'$. Then we first consider the lattices $L\subseteq L_2\subseteq L'$, define $\phi_{L,L_2}=i^*(\phi_{L,L'})$, $\mu_{L,L_2}=\mu_{L,L'}$ Then consider the lattices $L\subseteq L_1\subseteq L_2$, define $\phi_{L_1,L_2}=p_*(\phi_{L,L_2})$ and $\mu_{L_1,L_2}$ from the decomposition of $\mu_{L,L'}$ via the isomorphism (3.4). As explained above, the element $\phi_{L_1,L_2}\otimes\mu_{L_1,L_2}$ is well-defined. Then we obtain a compatible family $\{\phi_{L_1,L_2}\otimes\mu_{L_1,L_2}\}$, whose compatibility follows from the compatibility on $\CN$. 
\par We will often make use of the following net \begin{equation}
	\{(t^NL_0,t^{-M}L_0):M,N\in\BZ_{\geq0}\}
\end{equation}
An element labelled by this net is usually denoted $$(f_{N,-M}\otimes\mu_{N,-M})_{M,N\in\BZ_{\geq0}}$$ with $f_{N,-M}\in\CS(t^{-M}L_0/t^NL_0)$ and $\mu_{N,-M}\in\mu(t^NL_0,L_0)$.

\subsection{The Loop Group $\widetilde\GL_n(F((t)))\ltimes\sigma(F^*)$.} Let \begin{equation}
	\GL(X,L_0):=\{g\in\GL_F(X):\text{$L_0g$ and $L_0$ are commesurable}\}
\end{equation}
 Clearly $\GL(X,L_0)$ is a group. Consider the set
 \begin{equation}
	\widetilde\GL(X,L_0)=\{(g,\nu):g\in\GL(X,L_0),\,\nu\in\mu(L_0g^{-1},L_0)\}
\end{equation}
It admits a multiplication defined by \begin{equation}
	(g_1,\nu_1)(g_2,\nu_2)=(g_1g_2,(\nu_2g_1^{-1})\nu_1)
\end{equation}
where $\nu_2g_1^{-1}$ is the image of $\nu_2$ under the map $$g_1^{-1}:\mu(L_0g_2^{-1},L_0)\to\mu(L_0g_2^{-1}g_1^{-1},L_0g_1^{-1})$$ and
the product with $\nu_1$ means the image under the canonical isomorphism $$\mu(L_0g_2^{-1}g_1^{-1},L_0g_1^{-1})\times_{\BR_{>0}}\mu(L_0g_1^{-1},L_0)\to\mu(L_0(g_1g_2)^{-1},L_0)$$
\par The kernel of the canonical projection $\widetilde\GL(X,L_0)\to\GL(X,L_0)$ is the central subgroup $\{(1,\nu):\nu\in\mu(L_0,L_0)\}$ of $\widetilde\GL(X,L_0)$ which is isomorphic to the multiplicative group $\BR_{>0}$, so $\widetilde\GL_n(X,L_0)$ is a central extension of $\GL_n(X,L_0)$ by $\BR_{>0}$. See \cite{Li} for similar construction.
\par Clearly $\GL_n(F((t)))$ is a subgroup of $GL(X,L_0)$. For each $q\in F^*$, the map
$$\sigma(q):F^n((t))\to F^n((t)),\,\sum_{k\in\BZ}a_kt^k\mapsto\sum_{k\in\BZ}a_kq^{-k}t^k$$
is also an element in $\GL(X,L_0)$, and the map $\sigma:F^*\to\GL(X,L_0)$ is an injective homomorphism, whose image $\sigma(F^*)$ normalizes $\GL_n(F((t)))$. Since the central extension canonically splits over $\sigma(F^*)$, we can view it as a subgroup of $\widetilde\GL(X,L_0)$ still denoted $\sigma(F^*)$. Let $\widetilde\GL_n(F((t)))$ be the preimage of $\GL_n(F((t)))$ in the central extension $\widetilde\GL(X,L_0)$. We will prove
 $\widetilde\GL(X,L_0)$  isomorphic to the group $\widetilde\GL_n(F((t)))$ defined in Section 2.1.
We can form the semi-direct product $\widetilde\GL_n(F((t)))\ltimes\sigma(F^*)$, which is a subgroup of $\widetilde\GL_n(F((t)))$. The (right) action of $\sigma(F^*)$ on $\widetilde\GL_n(F((t)))$ is given by the formula
$$\sigma(q^{-1})(g(t),\nu)\sigma(q)=(g(q^{-1}t),\mu),\,q\in F^*,g(t)\in\GL_n(F((t))),\nu\in\mu(L_0g(t),L_0)$$
\begin{lem}
	The central extension $\widetilde\GL_n(F((t)))$ splits over the subgroup $\GL_n(\CO((t)))$ of $\GL_n(F((t)))$, where $\CO$ is the ring of integers of $F$.
\end{lem}
\begin{rmk}
	$\widetilde\GL_n(F((t)))$ is isomorphic to the standard central extension of $\GL_n(F((t)))$ given by a Tame symbol in the sense of Steinberg, c.f. \cite{Li}.
\end{rmk}
\subsection{The Representation of $\GL(X,L_0)$ on $\CS(X,L_0)$.}
We define an action of $\widetilde\GL_n(F((t)))$ on $\CS(X,L_0)$ as follows: for $(g,\nu)\in\widetilde\GL_n(F((t)))$ and $\phi=(\phi_{L,L'}=f_{L,L'}\otimes\mu_{L,L'})_{L\subseteq L'}\subseteq\CS(V,L_0)$, we define $\pi(g,\nu)\phi$ to be the family $(h_{L,L'}\otimes\gamma_{L,L'})_{L\subseteq L'}$ with
\begin{equation}
	h_{L,L'}(x)=f_{Lg^{-1},L'g^{-1}}(xg),\,\gamma_{L,L'}=(\mu_{Lg^{-1},L'g^{-1}}g)\nu
\end{equation}
 One can verify that the family $(h_{L,L'}\otimes\gamma_{L,L'})_{L\subseteq L'}$ is still compatible, namely $\pi(g,\nu)$ is well-defined.
\begin{thm}
	$\pi$ is a representation of $\widetilde\GL(X,L_0)$ on $\CS(X,L_0)$.
\end{thm}
\begin{proof}
	This boils down to checking $\pi(g_1g_2)$ acts the same as $\pi(g_1)\pi(g_2)$, which can be done by a direct computation.
\end{proof}
Next we will to establish an isomorphism of $\pi$ with part of ${\cal S} (  V_-)$ in Section 3.2 and at the same time prove
 the isomorphism of two groups $\widetilde\GL_n(F((t)))$, one defined in Section 3.2 and the other defined in this section above. 
\par Let $W$ be the standard $2n$-dimensional symplectic space $F^{2n}$ over $F$,  with the symplectic form defined by
  \[ \pair{ (x_1,y_1),(x_2,y_2)} = x_1y_2^t-x_2y_1^t.\]
  \par  Let $W((t))$ , $ V_+ $, $V_- $ be as in Section 2.1, then $W((t)) $ is sympletic space as in (2.1), and ${\cal S} ( V_-) = \CS(F^n[t^{-1}]t^{-1}\oplus F^n[t^{-1}]t^{-1}))$. Let $W((t))=X\oplus Y$, where $X,Y\cong F^n((t))$ are the first and last $n$-copies of $F((t))$. As is indicated by this notation, we will identify the space $X=F^n((t))$ in this section with the first $n$-copies of $F((t))$ in $W((t))$. Let $Y_-=F^n[t^{-1}]t^{-1}\subseteq Y$. 
\par We define an isomorphism
  $$\Phi:\CS(F^n[t^{-1}]t^{-1}\oplus F^n[t^{-1}]t^{-1}))\to\CS(X,L_0)$$ as follows: take $f\in \CS(F^n[t^{-1}]t^{-1}\oplus F^n[t^{-1}]t^{-1})$, we are going to define a compatible family $((\Phi f)_{L,L'}\otimes\mu_{L,L'})$ indexed by the net $$\{(L,L')\in\Gr_n(F)\times\Gr_n(F):L\subseteq L_0\subseteq L'\}$$
   We first identify $L'/L$ with $L'/L_0\oplus L_0/L$ as follows: for $x\in L'$, write $x=x_-+x_+$ with $x_-\in F^n[t^{-1}]t^{-1}$ and $x_+\in F^n[[t]]$, then the identification sends $x+L\in L'/L$ to $(x_-+L_0)\oplus(x_++L)$. The symplectic form on $W((t))$ defines a perfect pairing  $\pair{-,-}: Y_-\times L_0\to F$. Define
   $$L^\perp=\{y\in Y_-:\pair{y,v}=0,\,\forall v\in L\}$$ Each $y\in L^\perp$ defines a linear functional on $L_0/L$ by the pairing, which gives a linear map $L^\perp\to (L_0/L)^*$. Since the pairing $\pair{-,-}: Y_-\times L_0\to F$ is perfect, this is a linear isomorphism, and the induced pairing on $L^\perp\times (L_0/L)$ is a perfect pairing. In particular, $L^\perp$ is finite-dimensional. 
\par We fix a non-zero additive character $\psi:F\to U(1)$ as in Section 2.1,
     then the locally compact abelian groups $L^\perp$ and $L_0/L$ are in duality under $(y,v)\mapsto\psi(\pair{y,v})$, so the $\BR_{>0}$-torsor $\mu(L^\perp)$ of Haar measures on $L^\perp$ is dual to the $\BR_{>0}$-torsor $\mu(L,L_0)$. We choose $\mu_{L,L'}\in\mu(L,L_0)$, let $\mu^*\in\mu(L^\perp)$ be the dual Haar measure to $\mu_{L,L'}$. Then we define
   \begin{equation}
	(\Phi f)_{L,L'}(x_-,x_+)=\int_{L^\perp}f(x_-+y)\psi(-\pair{y,x_+})d\mu^*(y)
\end{equation}
Note that this is just the partial Fourier transform of $f$ on $L^\perp$. 
\par To simplify the exposition, we make use of the net (3.8), so given $f\in\CS(F^n[t^{-1}]t^{-1}\oplus F^n[t^{-1}]t^{-1})$, we take $L=t^{N}L_0,L'=t^{-M}L_0$, then we can identify $L'/L$ with the $F$-vector space $t^{-M}F^n\oplus\cdots\oplus t^{N-1}F^n$ by
$$L'/L\to V,\,x=x_{-M}t^{-M}+\cdots+x_{N-1}t^{N-1}+t^NL_0\mapsto (x_{-M},\cdots,x_{N-1})$$ with $x_i\in V$.
The decomposition $x=x_-+x_+$ is given by
$$x_-=(x_{-M},\cdots,x_{-1}),\,x_+=(x_0,\cdots,x_{N-1})$$
 in the same spirit. $L^\perp$ is identified with the space $t^{-N}F^n\oplus\cdots\oplus t^{-1}F^n$ whose coordinates are denoted $(y_{-N},\cdots,y_{-1})$. Then the definition becomes
	\begin{equation}
	\begin{split}
		&(\Phi f)_{N,-M}(x_{-M},\cdots,x_{N-1})=\int_{F^{nN}}f\left(\sum_{i=-M}^{-1}x_{i}t^i,\sum_{j=-N}^{-1}y_jt^j\right)\cdot\\&\psi(-(x_0\cdot y_{-1}+\cdots+x_{N-1}\cdot y_{-N}))dy_1\cdots dy_N
	\end{split}
	\end{equation}
	where $\cdot$ stands for the usual inner product on $F^n$. 
\par It is clear that $\Phi $ is a linear isomorphism. We will prove that $\Phi$ interwines the Weil representation on $\CS(V_-)$ and the action of $\widetilde\GL_n(F((t)))$ on $\CS(V_-)$ defined in this section in the following sense: 
\par For every $(g,\nu)\in\widetilde\GL_n(F((t)))$, we have $a(g)=\matrixx{g}{}{}{g^{-t}}\in\Sp_{2n}(F((t)))$, so we would like to compare $\pi(g)$ with the operator $T(a(g))$ given by Weil representation in Section 2.1. The definition of the projective Weil representation $T(h)$ for $h\in\Sp_{2n}(F((t)))$ depends on a choice of Haar measure on $\Im\gamma_h$, and this is done by taking natural choices of Haar measure for $h\in P$ and for $h$ is a diagonal matrix with diagonal entries powers of $t$ (see Section 3.2 or \cite{Z}). So in parallel, for every $(g,\nu)$ we can define a Haar measure on $\Im\gamma_{a(g)}$ by doing so for $g\in\GL_n(F[[t]])$ and for $g=\diag(t,1,\cdots,1)$. In the first case, $\Im\gamma_{a(g)}=\{0\}$ is a single point, and $\nu\in\mu(L_0,L_0)$ is a measure on a single point, so we just take the same measure on $\Im\gamma_{a(g)}$ as the one given by $\nu$; in the second case we have a natural isomorphism $\Im\gamma_{a(g)}\cong F$ (see \cite{Z}) and $\nu$ is a measure on a vector space canonically isomorphic to $F$, so we also assign the measure determined by $\nu$. With this choice of Haar measures, we can define an operator $T(g,\nu)$ on $\CS(V_-)$ by formula (2.6) for any $(g,\nu)\in\widetilde\GL_n(F((t)))$. We will prove below that this operator coincides with $\pi(g,\nu)$, namely the space $\CS(X,L_0)$ considered in this section is indeed a model for the Weil representation defined in Section 2.1 when restricted to the subgroup given by $\GL_n(F((t)))$. 
\begin{thm}
	The following diagram commutes for every $ (g,\nu) \in \widetilde\GL_n(F((t)))$:
\begin{center}
	\begin{tikzcd}
\CS(V_-) \arrow[r, "\Phi"] \arrow[d, "T({g,\nu})"'] & {\CS(X,L_0)} \arrow[d, "\pi({g,\nu})"] \\
\CS(V_-) \arrow[r, "\Phi"]                              & {\CS(X,L_0)}                    
\end{tikzcd}
\end{center}
\end{thm}
\begin{proof}
	Take $f\in\CS(V_-)$. It suffices to verify the commutativity for $g\in\GL_n(F[[t]])$ and for $g=\diag(t,1,\cdots,1)$. For simplicity, we fix $\nu$ so that the induced measure on $\Im\gamma_{a(g)}$ is the choice of Haar measure discussed in Section 3.2. So in the first case $\nu$ is the measure giving a point measure $1$, and in the second case $\nu$ is the usual Haar measure on the local field $F$. We will suppress $\nu$ from notations in the sequel. 
	\par We first prove the case $g=\diag(t,1,\cdots,1)$, which reduces to an $\SL_2$-computation, so we can assume $n=1$, then $g=t$ and $a(g)=\matrixx{t}{}{}{t^{-1}}$. By formula (3.14), for $f\in\CS(V_-)$, we have $$(\Phi f)_{N,-M}(x_{-M},\cdots,x_{N-1})=(\CF_{0}^{-1}\CF_{1}^{-1}\cdots\CF_{N-1}^{-1}f)(x_{-M},\cdots,x_{-1}|x_0,\cdots,x_{N-1})$$

Here and in the sequel we use freely the notations as in Section 4.2 and \cite{Z}. So \begin{align}
	&\nonumber(\pi(t)\Phi f)_{N,-M}(x_{-M},\cdots,x_{N-1})\\\nonumber=&(\Phi f)_{N+1,-M+1}(x_{-M},\cdots,x_{N-1})\\\nonumber=&(\CF_0^{-1}\CF_1^{-1}\cdots\CF_{N}^{-1}f)(x_{-M},\cdots,x_{-2}|x_{-1},\cdots,x_{N-1})\\=&(\tau_1\CF_0^{-1}\CF_1^{-1}\cdots\CF_{N}^{-1}f)(x_{-M},\cdots,x_{-1}|x_{0},\cdots,x_{N-1})
\end{align}On the other hand, by \cite{Z} Section 3 we have $$T\matrixx{t}{}{}{t^{-1}}=\CF_{-1}^{-1}\tau_1$$ so \begin{align}
	&\nonumber(\Phi(T\matrixx{t}{}{}{t^{-1}}f))_{N,-M}(x_{-M},\cdots,x_{N-1})\\\nonumber=&(\CF_{0}^{-1}\CF_{1}^{-1}\cdots\CF_{N-1}^{-1}T\matrixx{t}{}{}{t^{-1}}f)(x_{-M},\cdots,x_{-1}|x_0,\cdots,x_{N-1})\\=&(\CF_{0}^{-1}\CF_{1}^{-1}\cdots\CF_{N-1}^{-1}\CF_{-1}^{-1}\tau_1f)(x_{-M},\cdots,x_{-1}|x_0,\cdots,x_{N-1})
\end{align} The equality between (3.15) and (3.16) follows from the commutativity of the operators $\CF_i$ and the relation between $\tau_i$ and $\CF_j$ on p. 27 of \cite{Z}. 
\par Next we turn to the case $g\in\GL_n(F[[t]])$. Consider the $F$-valued symmetric bilinear form on $F^n((t))$ given by \begin{equation}
	(x(t),y(t))=\Res_{t=0}(x(t)y(t)^tdt)
\end{equation} Note that we view elements in $F^n((t))$ as row vectors. $F^n((t))$ admits a Lagrangian decomposition $$F^n((t))=F^n[t]\oplus F^n[t^{-1}]t^{-1}$$ So for every $g\in\GL_n(F((t)))$ we can write $g=\matrixx{a}{b}{c}{d}$ with respect to this decomposition, where $a:F^n[t^{-1}]^{-1}\to F^n[t^{-1}]t^{-1}$, $b:F^n[t^{-1}]t^{-1}\to F^n[t]$, $c:F^n[t]\to F^n[t^{-1}]t^{-1}$, $d:F^n[t]\to F^n[t]$. Note that for $g\in\GL_n(F[[t]])$, we have $c=0$. So we suppose $g=\matrixx{a}{b}{}{d}$.
\par Since $g\in\GL_n(F[[t]])$, for any lattice $L$ in the net (3.8) we have $Lg=L$, then by formula (5.3.12), for lattices $L\subseteq L_0\subseteq L'$ in the net (3.8) we have \begin{align}
	&\nonumber (\pi\matrixx{a}{b}{}{d}\Phi f)_{L,L'}(x_-,x_+)\\\nonumber=&(\Phi f)_{L,L'}(x_-a,x_-b+x_+d)\\=&\int_{L^\perp}\psi(-\pair{y,x_-b+x_+d})f(x_-a,y)d(yd^*)
\end{align} where we need to explain the measure $d(yd^*)$ on $L^\perp$. We suppose $L=t^NL_0$, $L'=t^{-M}L_0$, then by the definition of $\Phi f$, we can choose the measure $\mu_{L,L'}$ in the system $\Phi f$ to be $d(xd)$, where $dx$ is the usual Haar measure on $L_0/L\cong F^{nN}$, and $d(xd)$ is the Haar measure on $L_0/L$ induced by $d:F^n[t]\to F^n[t]$. After applying the operator $\pi\matrixx{a}{b}{}{d}$, this measure is changed by $d^{-1}$ by (3.12) since $Lg=L$, so becomes the usual Haar measure on $L_0/L$, which is fixed from the beginning of this proof. Then by the definition of $\Phi$, the resulting measure on $L^\perp\cong (L_0/L)^*$ is the dual Haar measure of $d(xd)$, namely $d(yd^*)$. 
\par To compute the other composition in the commutative diagram, we need a formula for $T(a(g))=T\matrixx{g}{}{}{g^{-t}}$, so we need a decomposition $a(g)=\matrixx{\alpha}{\beta}{\gamma}{\delta}$ with respect to $W((t))=V_-\oplus V_+$ as in (2.4). Since $g\in\GL_n(F[[t]])$, we clearly have $\gamma=0$. We first give a formula for $g^t$. $g^t$ is the unique $F$-linear transformation on $F^n((t))$ commuting with $t$ such that $$(xg,y)=(x,yg^t)$$ for all $x,y\in F^n((t))$, so a computation shows that \begin{equation}
	g^t=\matrixx{d^*}{b^*}{}{a^*}
\end{equation} Here the adjoints $a^*,b^*,d^*$ are the adjoints with respect to the bilinear form $(-,-)$ similar to those described in \cite{Z}, Section 2. Thus we have \begin{equation}
	g^{-t}=\matrixx{d^{-*}}{-d^{-*}b^*a^{-*}}{}{a^{-*}}
\end{equation} Here $^{-*}$ means the inverse of the adjoint, or adjoint of the inverse, they should be equal. 
\par As a result, we have $$\alpha=a+d^{-*},\,\beta=b-d^{-*}b^*a^{-*},\,\delta=d+a^{-*}.$$ So by formula (2.6), we have \begin{align}
	&\nonumber(T(a(g))f)(x_1,x_2)\\\nonumber=&\psi(\frac 1 2\pair{x_1a+x_2d^{-*},x_1b-x_2d^{-*}b^*a^{-*}})f(x_1a,x_2d^{-*})\\\nonumber=&\psi(-\frac 1 2\pair{x_1a,x_2d^{-*}b^*a^{-*}}-\frac 1 2\pair{x_1b,x_2d^{-*}})f(x_1a,x_2d^{-*})\\=&\psi(-\pair{x_1b,x_2d^{-*}})f(x_1a,x_2d^{-*})
\end{align} Thus we have \begin{align}
	&\nonumber (\Phi T(a(g))f)_{L,L'}(x_-,x_+)\\\nonumber=&\int_{L^\perp}\psi(-\pair{y,x_+})T(a(g))f(x_-,y)dy\\=&\int_{L^\perp}\psi(-\pair{y,x_+}-\pair{x_-b,yd^{-*}})f(x_-a,yd^{-*})dy
\end{align} It suffices to argue that (3.18) coincides with (3.22). To do this, we substitute $y$ with $yd^{-*}$ in (3.18), then (3.18) becomes \begin{align*}
	&\int_{L^\perp}\psi(-\pair{yd^{-*},x_-b+x_+d})f(x_-a,yd^{-*})dy\\=&\int_{L^\perp}\psi(-\pair{yd^{-*},x_-b}-\pair{y,x_+})f(x_-a,yd^{-*})dy
\end{align*} which is exactly (3.22). 
\end{proof}
In particular, this theorem indicates that there should be an action of a symplectic version of $\widetilde\GL(X,L_0)$ on the space $S(X,L_0)$. 

\subsection{Representation of Adelic group and Theta Functional. }We define  $\widetilde\GL_n(\BA\pair{t})$ to be the restricted direct product $\prod_v'\widetilde\GL_n(\BQ_v((t)))$
 with respect to the subgroups $\GL_n(\BZ_p((t)))$ for finite primes $p$. We have a short exact sequence
 \begin{equation}
	1\to\oplus_v\BR_{>0}\to\widetilde\GL_n(\BA\pair{t})\to\GL_n(\BA\pair{t})\to 1
\end{equation}
   It has a representation on the restricted tensor products of local representations $\CS_p(X,L_0)$ defined as follows.
   If $v=p$ is a finite place, we have a distinguished element $\bold1_p\in\CS(X_p,L_{0,p})$ given as follows:
   it suffices to define an element in $\CS(X_p,L_{0,p})$ on the net (3.8), so we define $f_{N,-M}\otimes\mu_{N,-M}\in\CS(X_p,L_{0,p})\otimes\mu(t^NL_{0,p},L_{0,p})$ by setting $f_{N,-M}\in\CS(t^{-M}L_{0,p}/t^NL_{0,p})\cong\CS(\BQ_p^{n(M+N)})$ to be the characteristic function of $\BZ_p^{n(M+N)}\subseteq\BQ_p^{n(M+N)}$ and $\mu_{N,-M}\in\mu(t^NL_{0,p},L_{0,p})=\mu(L_{0,p}/t^NL_{0,p})\cong\mu(\BQ_p^{nN})$ to be the Haar measure giving the set $\BZ_p^{nN}\subseteq\BQ_p^{nN}$ volume $1$.
    Let $\CS_{\BA}(X,L_0)$ be the restricted tensor product of the spaces $\CS_v(X,L_0)$ over the places $v$ of $\BQ$ with respect to the elements $\bold1_p\in\CS_p(X,L_0)$ for finite primes $p$.
    Note that $\bold1_p = \Phi ( \phi_{0, p } ) $ defined in Section 2.2.

\par Let $X_\BA=X\otimes_\BQ\BA$ and $L_{0,\BA}=L_0\otimes_\BQ\BA$. A \textbf{lattice} $L$ in $X$ is a free $\BQ\pair{t}_+$-submodule of rank $n$ in $X$ which is commeasurable with $L_0$. By a slight abuse of notations, let $\Gr_n(\BQ)$ be the set of lattices in $X$. Given an element $f\in\CS_\BA(X,L_0)$, for a pair of lattices $M\subseteq M'$ in $\Gr_n(\BQ)$, we will define a function $f_{M_\BA,M'_\BA}\in\CS(M'_\BA/M_\BA)$, where $\CS(M'_\BA/M_\BA)$ is the space of Schwartz functions on the free $\BA$-module $M'_\BA/M_\BA\cong (M'/M)_\BA$ of finite rank, as follows. 
\par We can assume $f=\otimes_v f_v$ is a pure tensor, where each local component is given by $f_v=(f_{v;L,L'}\otimes\mu_{v;L,L'})$ for $L\subseteq L'$ lattices in $X_v$. First we assume that $M\subseteq L_0$, then each $\mu_{v;M_v,L_{0,v}}$ is a Haar measure on $L_{0,v}/M_v\cong (L_0/M)_v$, so the product $\prod_{v}\mu_{v;M_v,L_{0,v}}$ is a Haar measure on $(L_0/M)_\BA$. Recall that for each finite-dimensional $\BQ$-vector space $K$ there is a canonical choice of Haar measures on $K_\BA$, determined by the condition that the quotient $K_\BA/K_\BQ$ has volume $1$. We say a choice of the measures $\mu_{v;M_v,L_{0,v}}$ for the various $v$ is \textbf{good} if the induced product measure on $(L_0/M)_\BA$ is the canonical measure. If $M\notin L_0$, there is an analogous notion of a good family of elements $\mu_{v;M_v,L_{0,v}}\in\mu(M_v,L_{0,v})$. Now we assume $\mu_{v;M_v,L_{0,v}}$ is a good family, then we define $f_{M_\BA,M'_\BA}$ to be the product $\prod_v f_{v;M_v,M_v'}$. This is well-defined, and the resulting family $(f_{M_\BA,M'_\BA})_{M,M\in\Gr_n(\BQ),\,M\subseteq M'}$ satisfies a compatibility relation analogous to (3.7). 
\par  Let
  \[ \CS_\BA'(X,L_0) =   \Phi ( {\cal S}' ( V_{ - , \BA} ) )     \]
 We	define the \textbf{theta functional} $\theta : \CS_\BA'(X,L_0) \to {\BC} $ by 
 \[  \theta ( \phi ) = \theta ( \Phi^{-1}  \phi )\]
 where  $\theta $ in the right hand side is given in Section 2.2.
\par To find $\theta ( \phi  )$ for $\phi \in \CS_\BA'(X,L_0) $ without using $\Phi^{-1} $,  we can proceed as follows: for $\phi=(\phi_{L,L'}\otimes\mu_{L,L'})_{L\subseteq L}\in\CS_\BA'(W,L_0)$, for $M,N\in\BZ_{\geq0}$, let $f_{N,-M}$ be the function $f_{t^NL_{0,\BA},t^{-M}L_{0,\BA}}\in\CS((t^NL_0/t^{-M}L_0)_\BA)$ defined above.  There is a partial theta functional $$\theta_{N,-M}:\CS((t^{-M}L_0/t^NL_0)_\BA)\to\BC,\,\theta_{N,-M}(f) =\sum_{r\in (t^NL_0/t^{-M}L_0)_\BQ}f(r)$$
 Then the limit
  $$\lim_{M,N\to\infty}\theta_{N,-M}(f_{N,-M})$$
   is $\theta ( \phi ) $. To see this, we start from $f\in \CS(V_{-,\BA})$, so that \begin{align*}
   	&(\Phi f)_{N,-M}(x_{-M},\cdots,x_{N-1})=\int_{\BA^{nN}}f(x_{-M},\cdots,x_{-1}|y_0,\cdots,y_{N-1})\\&\psi(x_0y_0+\cdots+x_{N-1}y_{N-1})dy_0\cdots dy_{N-1}
   \end{align*} Then \begin{align*}
   	&\theta_{N,-M}((\Phi f)_{N,-M})\\=&\sum_{r_{-M},\cdots,r_{N-1}\in\BQ^n}\int_{\BA^{nN}}f(r_{-M},\cdots,r_{-1}|y_0,\cdots,y_{N-1})\psi(r_0y_0+\cdots+r_{N-1}y_{N-1})dy\\=&\sum_{r_{-M},\cdots,r_{N-1}\in\BQ^n}f(r_{-M},\cdots,r_{-1}|r_0,\cdots,r_{N-1})&
   \end{align*} by Poisson summation formula. Taking limit as $M,N\to\infty$, it is clear that the limit is $$\theta(f)=\sum_{r\in\BQ^{2n}[t^{-1}]t^{-1}}f(r)$$ hence the result. 
\par In particular, the functional $\theta $ is $\GL_n(\BQ\pair{t})\ltimes\sigma(\BQ^*)$-invariant.
\section{Theta Lifting}
\subsection{The Dual Pair}Let $X=M_n(\BQ\pair{t}),\,L_0=M_n(\BQ\pair{t}_+)$. There is a right action of the group $G(\BA)=\GL_n(\BA\pair{t})$ on $X_\BA$ by $x(t)\cdot g=g^tx(t)$ (here $g^t$ is the transpose of $g$) which gives an injection $G(\BA)\hookrightarrow\GL_{n^2}(\BA\pair{t})$. In section 1.5 we constructed a central extension $\widetilde\GL_{n^2}(\BA\pair{t})$ of $\GL_{n^2}(\BA\pair{t})$ and a representation of this central extension on $\CS_\BA(X,L_0)$. It turns out that the pullback of the central extension $\widetilde\GL_{n^2}(\BA\pair{t})$ along the above injection is isomorphic to the central extension $\widetilde\GL_n(\BA\pair{t})$, so we can view $\widetilde\GL_n(\BA\pair{t})$ as a subgroup of $\widetilde\GL_{n^2}(\BA\pair{t})$, denoted $\widetilde G(\BA)$. Note that the central extension $\widetilde G(\BA)$ of $G(\BA)$ splits over the subgroup $G(\BQ):=\GL_n(\BQ\pair{t})$, so we can view $G(\BQ)$ as a subgroup of $\widetilde G(\BA)$.
\par We also consider the right action of $\GL_n(\BA\pair{t}_+)$ on $X_\BA$ by $x\cdot h=xh$ (matrix multiplication), which yields an injection $\GL_n(\BA\pair{t}_+)\hookrightarrow\GL_{n^2}(\BA\pair{t})$. It turns out that the central extension $\widetilde\GL_{n^2}(\BA\pair{t})$ splits over the image of this injection, so we can view $\GL_n(\BA\pair{t}_+)$ as a subgroup of $\widetilde\GL_{n^2}(\BA\pair{t})$, denoted $H(\BA)$. Let $H(\BA)^1=\GL_n(\BA\pair{t}_+)^1$ be the kernel of the composition map  $$\GL_n(\BA\pair{t}_+)\xrightarrow{\text{evaluation at }t=0}\GL_n(\BA)\xrightarrow{\det}\BA^*\xrightarrow{|\cdot|_\BA}\BR_{>0}$$
\par The two subgroups above will serve as an analogue of a ``dual pair'' in the classical theory of theta liftings. However, they are not commutative subgroups in $\widetilde\GL_{n^2}(\BA\pair{t})$ as in the classical case. The commutativity relation between them is that the subgroup $G(\BQ)$ of $\widetilde G(\BA)$ is commutative with $H(\BA)^1$ in $\widetilde\GL_{n^2}(\BA\pair{t})$, which ensures that the theta lifting of an automorphic function is still automorphic (here the word ``automorphic'' means to be left-invariant under the corresponding subgroup over $\BQ$.)
\subsection{Theta Liftings}Let $\GL_n(\BA)^1$ be the subgroup of $\GL_n(\BA)$ consisting of elements of determinant $1$. Take a classical cuspidal automorphic form $\varphi$ on $\GL_n(\BA)^1$. We also view $\varphi$ as a function on $H(\BA)^1=\GL_n(\BA\pair{t}_+)^1$ via the pullback along $$\GL_n(\BA\pair{t}_+)^1\xrightarrow{\text{evaluation at }t=0}\GL_n(\BA)^1$$ For a Schwartz function $f\in\CS'_\BA(M_n(\BQ\pair{t}),M_n(\BQ\pair{t}_+))$ the \textbf{theta lifting} of $\varphi$ is defined to be the function $\Theta_f(\varphi,g)$ on $g\in\widetilde G(\BA)=\widetilde\GL_n(\BA\pair{t})\ltimes\BA^\times$ given by \begin{equation}
	\Theta_f(\varphi,g)=\int_{\GL_n(\BQ\pair{t}_+)\backslash\GL_n(\BA\pair{t}_+)^1}\theta(\pi(h)\pi(g)f)\varphi(h)dh
\end{equation} where there exists a ``Haar measure'' on $\GL_n(\BQ\pair{t}_+)\backslash\GL_n(\BA\pair{t}_+)^1$ with total volume $1$. The integration is absolutely convergent by \cite{LZ2} Proposition 3.2. Also by commutativity of $G(\BQ)$ with $H(\BA)^1$ and the $G(\BQ)$-invariance of the theta functional, we see immediately that the resulting function $\Theta_f(\varphi,-)$ is left $G(\BQ)$-invariant, namely it is an automorphic function. 
\section{The case $n=1$.}In this section we compute the theta lifting of the constant function on $\GL_1(\BA\pair{t}_+)$, namely we are going to compute $$\Theta_f(\bold1,g)=\int_{\GL_1(\BQ\pair{t}_+)\backslash\GL_1(\BA\pair{t}_+)^1}\theta(\pi(h)\pi(g)f)dh$$ Let $F=\pi(g)f$, then by the definition of the theta functional, we have \begin{align}
	 &\Theta_f(\bold1,g)=\int_{\GL_1(\BQ\pair{t}_+)\backslash\GL_1(\BA\pair{t}_+)^1}\theta(\pi(h)F)dh\\\nonumber&=\int_{\GL_1(\BQ\pair{t}_+)\backslash\GL_1(\BA\pair{t}_+)^1}\lim_{M,N\to\infty}\sum_{\gamma\in t^{-M}\BQ\pair{t}_+/t^N\BQ\pair{t}_+}(\pi(h)F)_{N,-M}(\gamma)dh
\end{align}
\par We can exchange the order of taking limit and taking integration. This can be argued as follows: each of the $\theta_{N,-M}$ in the definition of the theta functional, when transported to the model $\CS(V_-)$ of Weil representation under $\Phi^{-1}$, is a sub-summation of the summation giving the theta functional on $\CS(V_-)$, so we can exchange the order by using the Fubini theorem and the fact that the theta integral (4.1) is absolutely convergent. 
\par So we consider \begin{equation}
	S_{N,-M}:=\int_{\GL_1(\BQ\pair{t}_+)\backslash\GL_1(\BA\pair{t}_+)^1}\sum_{\gamma\in t^{-M}\BQ\pair{t}_+/t^N\BQ\pair{t}_+}(\pi(h)F)_{N,-M}(\gamma)dh
\end{equation} Then $\Theta_f(\bold1,g)$ is the limit $\lim_{M,N\to\infty}S_{N,-M}$. We use the ``unfolding method'' to proceed the computation, namely we consider the $\GL_1(\BQ\pair{t}_+)$-action on the set $t^{-M}\BQ\pair{t}_+/t^N\BQ\pair{t}_+$. Clearly the orbits are classified the degree of an element, namely a set of orbit representatives are given as $t^r+t^N\BQ\pair{t}_+$ for $r=-M,-M+1,\cdots,N-1$. The stabilizer of $t^r+t^N\BQ\pair{t}_+$ is $$\GL_1(\BQ\pair{t}_+)_r=\{1+b_{N-r}t^{N-r}+b_{N-r+1}t^{N-r+1}+\cdots\in\BQ\pair{t}\}$$ So we have
\begin{align}
	 S_{N,-M}&=\sum_{r=-M}^{N-1}\int_{\GL_1(\BQ\pair{t}_+)_r\backslash\GL_1(\BA\pair{t}_+)^1}(\pi(h)F)_{N,-M}(t^r+t^N\BA\pair{t}_+)dh\\\nonumber&=\sum_{r=-M}^{N-1}\vol(\GL_1(\BQ\pair{t}_+)_r\backslash\GL_1(\BA\pair{t}_+)_r)\cdot\\\nonumber&\int_{\GL_1(\BA\pair{t}_+)_r\backslash\GL_1(\BA\pair{t}_+)^1}(\pi(h)F)_{N,-M}(t^r+t^N\BA\pair{t}_+)dh
\end{align}
where $\GL_1(\BA\pair{t}_+)_r=\{1+b_{N-r}t^{N-r}+b_{N-r+1}t^{N-r+1}+\cdots\in\BA\pair{t}\}$. Clearly $\vol(\GL_1(\BQ\pair{t}_+)_r\backslash\GL_1(\BA\pair{t}_+)_r)=1$, and since $h\in\GL_1(\BA\pair{t}_+)^1$ which preserves each of the local lattices $t^k\BQ_v[[t]]$ and maps a good family to another good family, we have $$(\pi(h)F)_{N,-M}(t^r+t^N\BA\pair{t}_+)=F_{N,-M}(t^rh+t^N\BA\pair{t}_+)$$ Thus $$S_{N,-M}=\sum_{r=-M}^{N-1}\int_{\GL_1(\BA\pair{t}_+)_r\backslash\GL_1(\BA\pair{t}_+)^1}F_{N,-M}(t^rh+t^N\BA\pair{t}_+)dh$$ Each term in this summation is actually the integration of $F_{N,-M}$ over the orbit of $t^r$ in $t^{-M}\BA\pair{t}_+/t^N\BA\pair{t}_+$ under $\GL_1(\BA\pair{t})^1$ with compatible measures. Clearly this orbit can be identified with the set $$\{a_0t^r+a_1t^{r+1}+\cdots+a_{N-r-1}t^{N-1}:a_0\in\BA^1,\,a_1,\cdots,a_{N-r-1}\in\BA\}$$ where $\BA^1$ is the subset of $\BA^*$ of elements of adele norm $1$, which has a standard measure. The standard measure on $\GL_1(\BA\pair{t}_+)_r\backslash\GL_1(\BA\pair{t}_+)^1$ transfers to the standard measure on the set $\BA^1\times\BA^{\oplus(N-r-1)}$. So we have
\begin{align}
	&S_{N,-M}\\\nonumber& =\sum_{r=-M}^{N-1}\int_{\BA^1}\int_{\BA^{\oplus(N-r-1)}}F_{N,-M}(a_0t^r+\cdots+a_{N-r-1}t^{N-1}+t^N\BA\pair{t}_+)d^\times a_0da_1\cdots da_{N-r-1}\\\nonumber&=\sum_{r=-M}^N\int_{\BA^1}\int_{\BA^{\oplus(N-r-1)}}F_{N,r}(a_0t^r+\cdots+a_{N-r-1}t^{N-1}+t^N\BA\pair{t}_+)d^\times a_0da_1\cdots da_{N-r-1}
\end{align} Because of the $i^*$-compatibility in (3.7) We set $$I(r,N)=\int_{\BA^1}\int_{\BA^{\oplus(N-r-1)}}F_{N,r}(a_0t^r+\cdots+a_{N-r-1}t^{N-1}+t^N\BA\pair{t}_+)d^\times a_0da_1\cdots da_{N-r-1}$$
\begin{lem}
	$I(r,N)=I(r,N+1)$
\end{lem}
\begin{proof} By the $p_*$-compatibility in (3.7) we have
	\begin{align*}
		I(r,N)&=\int_{\BA^1}\int_{\BA^{\oplus(N-r-1)}}F_{r,N}(a_0t^r+\cdots+a_{N-r-1}t^{N-1}+t^N\BA\pair{t}_+)d^\times a_0da_1\cdots da_{N-r-1}\\&=\int_{\BA^1}\int_{\BA^{\oplus(N-r-1)}}\int_\BA F_{r,N+1}(a_0t^r+\cdots+a_{N-r-1}t^{N-1}+a_{N-r}t^N+t^{N+1}\BA\pair{t}_+)\\&d^\times a_0da_1\cdots da_{N-r}\\&=I(r,N+1)
	\end{align*}
\end{proof}
\begin{thm}
	Define $I_f(x)=\int_{\BA^1}(\pi(x)f)_{1,0}(a_0+t\BA\pair{t}_+)d^\times a_0$ for $x\in\widetilde\GL_1(\BA\pair{t})$. Then we have $$\Theta_f(\bold1,g)=\sum_{r\in\GL_1(\BQ\pair{t})_+\backslash\GL_1(\BQ\pair{t})}I_f(r g)$$
\end{thm}
\begin{proof}
	Define $I_r=I(r,r+1)$, then $I_r=I(r,N)$ for all $N>r$, thus $$\Theta_f(\bold1,g)=\lim_{M,N\to\infty}S_{-M,N}=\sum_{r\in\BZ}I_r$$ Also we have $$I_r=\int_{\BA^1}F_{r+1,r}(a_0t^r+t^{r+1}\BA\pair{t}_+)d^\times a_0=\int_{\BA^1}(\pi(t^r)f)_{1,0}(a_0+t\BA\pair{t}_+)d^\times a_0$$ The result then follows from the fact that $\{t^r:r\in\BZ\}$ is a set of representatives of $\GL_1(\BQ\pair{t})_+\backslash\GL_1(\BQ\pair{t})$.
\end{proof}
\section{The case $n>1$}
In this section we compute the theta lifting of the function on $\GL_n(\BA\pair{t}_+)$ obtained from a classical cuspidal automorphic form $\varphi$ on $\GL_n(\BA)^1$ as explained in section 2.2, namely we are going to compute $$\Theta_f(\varphi,g)=\int_{\GL_n(\BQ\pair{t}_+)\backslash\GL_n(\BA\pair{t}_+)^1}\theta(\pi(h)\pi(g)f)\varphi(h)dh$$
The general strategy is similar to the one in the previous section, namely we first exchange limit and integration and consider summations analogous to $S_{N,-M}$, for which we use unfolding method to proceed. However, in the case $n>1$ the orbits are much more complicated. What saves the day is that the orbital integrals corresponding to a large part of the orbits vanishes because of cuspidality of $\varphi$. The rest of the orbits will be assembled to match the set $\Gr_n(\BQ)$ just as in the case $n=1$. Such a summation can be interpreted as an Eisenstein series.
\subsection{Step 1: Unfolding} We still let $F=\pi(g)f$, then by the definition of the theta functional, we have \begin{align}
	 &\Theta_f(\varphi,g)=\int_{\GL_n(\BQ\pair{t}_+)\backslash\GL_n(\BA\pair{t}_+)^1}\theta(\pi(h)F)\varphi(h)dh\\\nonumber&=\int_{\GL_n(\BQ\pair{t}_+)\backslash\GL_n(\BA\pair{t}_+)^1}\lim_{M,N\to\infty}\sum_{\gamma\in t^{-M}M_n(\BQ\pair{t}_+)/t^NM_n(\BQ\pair{t}_+)}(\pi(h)F)_{N,-M}(\gamma)\varphi(h)dh
\end{align}
We can still exchange the order of taking limit and integration , so $\Theta_f(\varphi,g)$ is equal to the limit as $M,N\to\infty$ of
\begin{equation}
	S_{N,-M}:=\int_{\GL_n(\BQ\pair{t}_+)\backslash\GL_n(\BA\pair{t}_+)^1}\sum_{r\in t^{-M}M_n(\BQ\pair{t}_+)/t^NM_n(\BQ\pair{t}_+)}F_{N,-M}(rh)\varphi(h)dh
\end{equation}
Let $U(\BA)$ be the kernel of the homomorphism $\GL_n(\BA\pair{t}_+)\to\GL_n(\BA)$ given by evaluation at $t=0$, let $U(\BQ)=U(\BA)\cap\GL_n(\BQ\pair{t}_+)$, then we have a semi-direct product $\GL_n(\BA\pair{t}_+)^1=U(\BA)\GL(\BA)^1$, so
$$S_{N,-M}=\int_{\GL_n(\BQ)\backslash\GL_n(\BA)^1}\int_{U(\BQ)\backslash U(\BA)}\sum_{r\in t^{-M}M_n(\BQ\pair{t}_+)/t^NM_n(\BQ\pair{t}_+)}F_{N,-M}(rha)\varphi(a)dhda$$
\par Now we use the unfolding method, consider the right $U(\BQ)$-action on the set $t^{-M}M_n(\BQ\pair{t}_+)/t^NM_n(\BQ\pair{t}_+$, which we identify with $M_n(\BQ\pair{t})_{N,-M}=t^{-M}M_n(\BQ)\oplus t^{-M+1}M_n(\BQ)\oplus\cdots\oplus t^{N-1}M_n(\BQ)$. For $r\in M_n(\BQ\pair{t})_{N,-M}/U(\BQ)$, let $U(\BQ)_{r,N}$ be the stabilizers of $r$ in $t^{-M}M_n(\BQ\pair{t}_+)/t^NM_n(\BQ\pair{t}_+)$, let $U(\BA)_{r,N}$ be the stabilizers of $r$ in $t^{-M}M_n(\BA\pair{t}_+)/t^NM_n(\BA\pair{t}_+)$ Then we have \begin{align}
	&\int_{U(\BQ)\backslash U(\BA)}\sum_{r\in t^{-M}M_n(\BQ\pair{t}_+)/t^NM_n(\BQ\pair{t}_+)}F_{N,-M}(rha)dh\\\nonumber&=\sum_{r\in M_n(\BQ\pair{t})_{N,-M}/U(\BQ)}\int_{U(\BQ)_{r,N}\backslash U(\BA)}F_{N,-M}(rha+t^NM_n(\BA\pair{t}_+))dh\\\nonumber&=\sum_{r\in M_n(\BQ\pair{t})_{N,-M}/U(\BQ)}\vol(U(\BQ)_{r,N}\backslash U(\BA)_{r,N})\int_{U(\BA)_{r,N}\backslash U(\BA)}F_{N,-M}(rha+t^NM_n(\BA\pair{t}_+))dh
\end{align}
We still have $\vol(U(\BQ)_{r,N}\backslash U(\BA)_{r,N})=1$. Also we  think of the integration $$\int_{U(\BA)_{r,N}\backslash U(\BA)}F_{N,-M}(rha+t^NM_n(\BA\pair{t}_+))dh$$ as doing integration on the orbit of $r$ under $U(\BA)$ with compatible measures. Note that we have a measure-preserving set bijection $U(\BA)=1+tM_n(\BA\pair{t}_+)$, so the orbit of $r$ under $U(\BA)$ is identified with $r+rtM_n(\BA\pair{t}_+)$ with compatible measures. Thus this integration is equal to $$\int_{\frac{rtM_n(\BA\pair{t}_+)+t^NM_n(\BA\pair{t}_+)}{t^NM_n(\BA\pair{t}_+)}}F_{N,-M}((r+\alpha)a)d\alpha=$$ $$\int_{\frac{rtM_n(\BA\pair{t}_+)+t^NM_n(\BA\pair{t}_+)}{t^NM_n(\BA\pair{t}_+)}}F_{t^NM_n(\BA\pair{t}_+),rM_n(\BA\pair{t}_+)+t^NM_n(\BA\pair{t}_+)}(ra+\alpha)d\alpha$$ because of the $i^*$-compatibility in (3.7). To summarize, so far we have \begin{equation}
	\begin{split}
		&S_{N,-M}=\int_{\GL_n(\BQ\pair{t}_+)\backslash\GL_n(\BA\pair{t}_+)^1}\sum_{r\in r\in M_n(\BQ\pair{t})_{N,-M}/U(\BQ)}\int_{\frac{rtM_n(\BA\pair{t}_+)+t^NM_n(\BA\pair{t}_+)}{t^NM_n(\BA\pair{t}_+)}}\\&F_{t^NM_n(\BA\pair{t}_+),rM_n(\BA\pair{t}_+)+t^NM_n(\BA\pair{t}_+)}(ra+\alpha)\varphi(a)d\alpha da
	\end{split}
\end{equation}
We need to unfold the outer integration again, consider the right $\GL_n(\BQ)$-action on $r\in M_n(\BQ\pair{t})_{N,-M}/U(\BQ)$. For any $r\in R_{N,-M}$ let $\GL_{n,r}(\BQ)$ be the stabilizer of $r$, which is always a unipotent radical (it may equal to $\{1_n\}$, where $1_n$ is the identity $n\times n$ matrix). Thus we have \begin{equation}
		S_{N,-M}=\sum_{r\in r\in M_n(\BQ\pair{t})_{N,-M}/\GL_n(\BQ\pair{t}_+)}I(r,N)
\end{equation} where 
\begin{equation}
	\begin{split}
		I(r,N)=&\int_{\frac{rtM_n(\BA\pair{t}_+)+t^NM_n(\BA\pair{t}_+)}{t^NM_n(\BA\pair{t}_+)}}\int_{\GL_{n,r}(\BQ\pair{t}_+)\backslash\GL_{n,r}(\BA\pair{t}_+)}\int_{\GL_{n,r}(\BA\pair{t}_+)\backslash\GL_n(\BA\pair{t}_+)^1}\\&F_{t^NM_n(\BA\pair{t}_+),rM_n(\BA\pair{t}_+)+t^NM_n(\BA\pair{t}_+)}(ra+\alpha)\varphi(ua)d\alpha du da
	\end{split}
\end{equation}
\subsection{Step 2: Analysis of Orbits}Let $R_{N,-M}$ be a set of representatives of $r\in M_n(\BQ\pair{t})_{N,-M}/\GL_n(\BQ\pair{t}_+)$. The definition of $I(r,N)$ is actually valid for $r\in M_n(\pair{t})_{N,-M}$ and is $\GL_n(\BQ\pair{t}_+)$-invariant. As we already mentioned in the beginning of this section, a large number of orbital integrals $I(r,N)$ actually vanishes because of cuspidality of $\varphi$. In the sequel we will find a set $C_N\subseteq M_n(\BQ\pair{t})$ such that $I(r,N)=0$ as long as $r\notin C_N$, so the summation for $S_{N,-M}$ can actually be taken on $C_N\cap R_{N,-M}$ (since the set $C_N$ is bi-$\GL_n(\BQ\pair{t}_+$-invariant by Lemma 6.2.3 below, the choice of representatives in $R_{N,-M}$ does not change the result). It turns out that for $r\in C_N\cap R_{N,-M}$ we have $I(r,N)=I(r,N+1)$ which is analogous to Lemma 5.0.1, and then after taking limit, the summation can be interpreted to be on the affine Grassmannian $\Gr_n(\BQ)$, so we can interpret the theta lifting as an Eisenstein series.
\begin{defn}
	Define $C_N\subseteq M_n(\BQ\pair{t})$ to be the set of $r\in M_n(\BQ\pair{t})$ satisfying the following condition: there does not exist a unipotent radical $U'\subseteq\GL_n(\BA)^1$ such that $$ru-r\in rtM_n(\BA\pair{t}_+)+t^NM_n(\BA\pair{t}_+),\,u\in U'(\BA)$$
\end{defn}
\begin{lem}
	Let $r$ be an element in $M_n(\BQ\pair{t})_{N,-M}$. If $r\notin C_N$, then $I(r,N)=0$.
\end{lem}
\begin{proof}
	If $r\notin C_N$, then there exists a unipotent radical $U'\subseteq\GL_2(\BA)^1$ such that $$ru-r\in rtM_n(\BA\pair{t}_+)+t^NM_n(\BA\pair{t}_+),\,u\in U'(\BA)$$ Recall that \begin{equation*}
	\begin{split}
		I(r,N)=&\int_{\frac{rtM_n(\BA\pair{t}_+)+t^NM_n(\BA\pair{t}_+)}{t^NM_n(\BA\pair{t}_+)}}\int_{\GL_{n,r}(\BQ\pair{t}_+)\backslash\GL_{n,r}(\BA\pair{t}_+)}\int_{\GL_{n,r}(\BA\pair{t}_+)\backslash\GL_n(\BA\pair{t}_+)^1}\\&F_{t^NM_n(\BA\pair{t}_+),rM_n(\BA\pair{t}_+)+t^NM_n(\BA\pair{t}_+)}(ra+\alpha)\varphi(ua)d\alpha du da
	\end{split}
\end{equation*} If $\GL_{n,r}$ is a non-trivial unipotent radical, then the above integration on $u\in\GL_{n,r}(\BQ\pair{t}_+)\backslash\GL_{n,r}(\BA\pair{t}_+)$ vanishes by cuspidality of $\varphi$, so we only need to deal with the case when $\GL_{n,r}=\{1_n\}$, where the above formula becomes 
\begin{align*}
	I(r,N)=&\int_{\GL_n(\BA\pair{t}_+)^1}\int_{\frac{rtM_n(\BA\pair{t}_+)+t^NM_n(\BA\pair{t}_+)}{t^NM_n(\BA\pair{t}_+)}}\\&F_{t^NM_n(\BA\pair{t}_+),rM_n(\BA\pair{t}_+)+t^NM_n(\BA\pair{t}_+)}(ra+\alpha)\varphi(a)d\alpha da
\end{align*} Now consider the following function in $a\in\GL_n(\BA)^1$ $$\int_{\frac{rtM_n(\BA\pair{t}_+)+t^NM_n(\BA\pair{t}_+)}{t^NM_n(\BA\pair{t}_+)}}F_{t^NM_n(\BA\pair{t}_+),rM_n(\BA\pair{t}_+)+t^NM_n(\BA\pair{t}_+)}(ra+\alpha)\varphi(a)d\alpha da$$ we claim that it is $U'(\BQ)$-invariant. Indeed, for $u'\in U'(\BQ)$ we have \begin{align*}
	&\int_{\frac{rtM_n(\BA\pair{t}_+)+t^NM_n(\BA\pair{t}_+)}{t^NM_n(\BA\pair{t}_+)}}F_{t^NM_n(\BA\pair{t}_+),rM_n(\BA\pair{t}_+)+t^NM_n(\BA\pair{t}_+)}(ru'a+\alpha)\varphi(u'a)d\alpha da\\&=\int_{\frac{rM_n(\BA\pair{t}_+)+t^NM_n(\BA\pair{t}_+)}{t^NM_n(\BA\pair{t}_+)}}F_{t^NM_n(\BA\pair{t}_+),rtM_n(\BA\pair{t}_+)+t^NM_n(\BA\pair{t}_+)}(ra+(ru'a-ra+\alpha))\varphi(a)da\\&=\int_{\frac{rM_n(\BA\pair{t}_+)+t^NM_n(\BA\pair{t}_+)}{t^NM_n(\BA\pair{t}_+)}}F_{t^NM_n(\BA\pair{t}_+),rtM_n(\BA\pair{t}_+)+t^NM_n(\BA\pair{t}_+)}(ra+\alpha)\varphi(a)d\alpha da
\end{align*} by a change of variable $\alpha\mapsto(ru'a-ra+\alpha)$ on $\frac{rtM_n(\BA\pair{t}_+)+t^NM_n(\BA\pair{t}_+)}{t^NM_n(\BA\pair{t}_+)}$. Thus we have \begin{align*}
	I(r,N)=&\sum_{u'\in U'(\BQ)}\int_{U'(\BQ)\backslash\GL_n(\BA)^1}\int_{\frac{rtM_n(\BA\pair{t}_+)+t^NM_n(\BA\pair{t}_+)}{t^NM_n(\BA\pair{t}_+)}}\\&F_{t^NM_n(\BA\pair{t}_+),rM_n(\BA\pair{t}_+)+t^NM_n(\BA\pair{t}_+)}(ra+\alpha)\varphi(a)d\alpha da\\&=\sum_{u'\in U'(\BQ)}\int_{U'(\BQ)\backslash U'(\BA)}\int_{U'(\BA)\backslash\GL_n(\BA)^1}\int_{\frac{rtM_n(\BA\pair{t}_+)+t^NM_n(\BA\pair{t}_+)}{t^NM_n(\BA\pair{t}_+)}}\\&F_{t^NM_n(\BA\pair{t}_+),rM_n(\BA\pair{t}_+)+t^NM_n(\BA\pair{t}_+)}(ra+\alpha)\varphi(u'a)d\alpha du'da
\end{align*} Again, the integration for $u'\in U'(\BQ)\backslash U'(\BA)$ vanishes by cuspidality of $\varphi$, thus $I(r,N)=0$. 
\end{proof}
Next we will give a concrete description of the set $C_N$.
\begin{lem}
	$r\in C_N$ if and only if $g_1rg_2+t^Nm\in C_N$ for some $g_1,g_2\in \GL_n(\BQ\pair{t}_+),\,m\in M_n(\BQ\pair{t}_+)$.
\end{lem}
\begin{proof}
	For $g\in \GL_n(\BQ\pair{t}_+)$, $r\notin C_N$ implies $gr\notin C_N$: the same unipotent radical also works for $gr$. As for right invariance, we show it for $h\in\GL_n(\BQ)$ and $k=1+a_1t+\cdots$ separately. For $h\in\GL_n(\BQ)$, if $r\notin C_N$, suppose $U'$ is the unipotent radical s.t. $ru-r\in rtM_n(\BA\pair{t}_+)+t^NM_n(\BA\pair{t}_+),\,u\in U'(\BA)$, then $h^{-1}U'h$ works for $rh$; for $k=1+a_1t+\cdots$ if $r\notin C_N$, suppose $U'$ is the unipotent radical s.t. $ru-r\in rtM_n(\BA\pair{t}_+)+t^NM_n(\BA\pair{t}_+),\,u\in U'(\BA)$, we have $$rk(k^{-1}uk)-rk\in rktM_n(\BA\pair{t}_+)+t^NM_n(\BA\pair{t}_+)$$ since $k^{-1}uk\in u+tM_n(\BA\pair{t}_+)$, this $U'$ also works for $rk$. Finally for $m\in M_n(\BQ\pair{t}_+)$, if $r\notin C_N$, suppose $U'$ is the unipotent radical s.t. $ru-r\in rtM_n(\BA\pair{t}_+)+t^NM_n(\BA\pair{t}_+),\,u\in U'(\BA)$, we have $$(r+t^Nm)u-(r+t^Nm)=(ru-r)+t^N(mu-m)\in rtM_n(\BA\pair{t}_+)+t^NM_n(\BA\pair{t}_+)$$ so $r+t^Nm\notin C_N$.
\end{proof}
\begin{prop}
	$$C_N=\bigsqcup_{a_1\leq \cdots\leq a_n<N}\GL_n(\BQ\pair{t}_+)\diag(t^{a_1},\cdots,t^{a_n})\GL_n(\BQ\pair{t}_+)$$ where $\diag(t^{a_1},\cdots,t^{a_n})$ is the diagonal matrix with diagonal entries $t^{a_1},\cdots,t^{a_n}$.
\end{prop}
\begin{proof}
	By the above lemma, $C_N$ is a union of $\GL_n(\BQ\pair{t}_+)$-double cosets. By an analogue of Cartan decomposition, each double coset has a representative of the form $\diag(\diag(t^{a_1},\cdots,t^{a_k},0,\cdots,0)$ for some $k\leq n$, and we only need examine these representatives. For $r=\diag(t^{a_1},\cdots,t^{a_k},0,\cdots,0)$ with $k<n$, take $U'=\{\matrixx{1_{n-1}}{}{*}{1}\}$, then $ru=r$ for $u\in U'$, so $r\notin C_N$. If $r=\diag(t^{a_1},\cdots,t^{a_n})$ with $a_1\leq\cdots\leq a_n$, if $a_n\geq N$, then $$r-t^N\matrixx{0_{n-1}}{}{}{t^{a_2-N}}=\diag(t^{a_1},\cdots,t^{a_{n-1}},0)\notin C_N$$ so $r\notin C_N$ by the above lemma.
	 \par It suffices to prove that $r=\diag(t^{a_1},\cdots,t^{a_n})\in C_N$ for $a_1\leq \cdots\leq a_n<N$, namely we need to prove that if $ru-r\in rtM_n(\BA\pair{t}_+)+t^NM_n(\BA\pair{t}_+)$, then $u=1$. Suppose $ru-r=rtm_1+t^Nm_2$ with $m_1,m_2\in M_n(\BQ\pair{t}_+)$, then $$u-1=tm_1+\diag(t^{N-a_1},\cdots,t^{N-a_n})m_2\in tM_n(\BQ\pair{t}_+)$$ but the degree of $u$ is $0$, this forces $u-1=0$.
\end{proof}
\begin{lem}
	If $r\in C_N\cap R_{N,-M}$, then $I(r,N)=I(r,N+1)$.
\end{lem}
\begin{proof}
	 If $r\in C_N$, then $\GL_{n,r}=\{1_n\}$ and $t^NM_n(\BQ\pair{t}_+)\subseteq rtM_n(\BQ\pair{t}_+)$, so
\begin{align*}
	 I(r,N)&=\int_{\GL_n(\BA\pair{t}_+)^1}\int_{\frac{rtM_n(\BA\pair{t}_+)}{t^NM_n(\BA\pair{t}_+)}}F_{t^NM_n(\BA\pair{t}_+),rM_n(\BA\pair{t}_+)}(ra+\alpha)\varphi(a)d\alpha da\\&=\int_{\GL_n(\BA\pair{t}_+)^1}\int_{\frac{rtM_n(\BA\pair{t}_+)}{t^NM_n(\BA\pair{t}_+)}}\int_{\frac{t^NM_n(\BA\pair{t}_+)}{t^{N+1}M_n(\BA\pair{t}_+)}}\\&F_{t^{N+1}M_n(\BA\pair{t}_+),rM_n(\BA\pair{t}_+)}(ra+\alpha+\beta)\varphi(a)d\alpha d\beta da\\&=I(r,N+1)
\end{align*} by the $p_*$-compatibility in (3.7). 
\end{proof}
So we define $I_r=I(r,N)$ for all $N$ such that $r\in C_N$. In particular, we have $$I_r=\int_{\GL_n(\BA)^1}F_{rtM_n(\BA\pair{t}_+),rM_n(\BA\pair{t}_+)}(ra)\varphi(a) da$$ by the compatibility conditions (3.7).
\begin{thm}
	 $$\lim_{M,N\to\infty}S_{N,-M}=\sum_{r\in\GL_n(\BQ\pair{t})/\GL_n(\BQ\pair{t}_+)}\int_{\GL_n(\BA)^1}F_{rM_n(\BA\pair{t}_+),rtM_n(\BA\pair{t}_+)}(ra)\varphi(a)da$$
\end{thm}
\begin{proof}Recall that $M_n(\BQ\pair{t})_{N,-M}=t^{-M}M_n(\BQ)\oplus t^{-M+1}M_n(\BQ)\oplus\cdots\oplus t^{N-1}M_n(\BQ)$, then we have
	\begin{align*}
		S_{N,-M}&=\sum_{r\in C_N\cap R_{N,-M}}I(r,N)\\&=\sum_{r\in (C_N\cap M_n(\BQ\pair{t})_{-M,N})/\GL_n(\BQ\pair{t}_+)}\int_{\GL_n(\BA)^1}f_{rtM_n(\BA\pair{t}_+),rM_n(\BA\pair{t}_+)}(r a)\varphi(a)da
	\end{align*}
	Note that $\GL_n(\BQ\pair{t})=\bigcup_{N\in\BZ}C_N$, hence the result.
\end{proof}
\begin{thm}
	Consider the function on $x\in\widetilde\GL_n(\BA\pair{t}_+)$ given by $$I_f(\varphi, x)=\int_{\GL_n(\BA)^1}(\pi(x)f)_{tM_n(\BA\pair{t}_+),M_n(\BA\pair{t}_+)}(a)\varphi(a)da$$ then $$\Theta_f(\varphi,g)=\sum_{r\in\GL_n(\BQ\pair{t}_+)\backslash\GL_n(\BQ\pair{t})}I_f(\varphi,rg)$$
\end{thm}
\begin{proof}
	\begin{align*}
		 &\Theta_f(\varphi,g)=\lim_{M,N\to\infty}S_{N,-M}\\&=\sum_{r\in\GL_n(\BQ\pair{t})/\GL_n(\BQ\pair{t}_+)}\int_{\GL_n(\BA)^1}F_{rtM_n(\BA\pair{t}_+),rM_n(\BA\pair{t}_+)}(ra)\varphi(a)da\\&=\sum_{r\in\GL_n(\BQ\pair{t})/\GL_n(\BQ\pair{t}_+)}\int_{\GL_n(\BA)^1}(\pi(r^t)F)_{tM_n(\BA\pair{t}_+),M_n(\BA\pair{t}_+)}(ra)\varphi(a)da\\&=\sum_{r\in\GL_n(\BQ\pair{t}_+)\backslash\GL_n(\BQ\pair{t})}\int_{\GL_n(\BA)^1}(\pi(r)F)_{tM_n(\BA\pair{t}_+),M_n(\BA\pair{t}_+)}(ra)\varphi(a)da\\&=\sum_{r\in\GL_n(\BQ\pair{t}_+)\backslash\GL_n(\BQ\pair{t})}\int_{\GL_n(\BA)^1}(\pi(r)\pi(g)f)_{tM_n(\BA\pair{t}_+),M_n(\BA\pair{t}_+)}(ra)\varphi(a)da\\&=\sum_{r\in\GL_n(\BQ\pair{t}_+)\backslash\GL_n(\BQ\pair{t})}I_f(\varphi,rg)
	\end{align*}
\end{proof}
By applying the definition of the map $\Phi$ in Section 3.4, we get Theorem 1.0.1.  
\par Note that $\GL_n(\BQ\pair{t}_+)$ can be viewed as a parabolic subgroup of $\GL_n(\BQ\pair{t})$ with unipotent radical $U(\BQ)$ and Levi subgroup $\GL_n(\BQ)$. The function $I_f(x)$ is $U(\BA)\GL_n(\BQ)$-invariant, so the theta lifting can be explained as an Eisenstein series induced from this function.

\end{document}